\documentclass[12pt,reqno]{amsart}
\usepackage{amsfonts,amscd,latexsym,amsmath,amssymb,cancel,enumerate,mathrsfs}
\theoremstyle{plain}
\newtheorem{master}{Master}[section]

\newtheorem{thm}[master]{Theorem}

\newtheorem{cor}[master]{Corollary}

\newtheorem{claim}[master]{Claim}
\theoremstyle{definition}
\newtheorem{defin}[master]{Definition}

\theoremstyle{remark}

\numberwithin{equation}{section}
\newcommand{\Ur}{\mathbb{U}}
\newcommand{\Rea}{\mathbb{R}}
\newcommand{\Nat}{\mathbb{N}}
\newcommand{\Int}{\mathbb{Z}}
\newcommand{\Rat}{\mathbb{Q}}

\begin{document}
\title[Non-universality of certain automorphism groups]{Non-universality of automorphism groups of uncountable ultrahomogeneous structures}
\author[M. Doucha]{Michal Doucha}
\address{Institute of Mathematics\\ Polish Academy of Sciences\\
00-656 Warszawa, Poland}
\email{m.doucha@post.cz}
\thanks{The author was supported by funds allocated to the implementation of the international co-funded project in the years 2014-2018, 3038/7.PR/2014/2, and by the EU grant PCOFUND-GA-2012-600415.}
\begin{abstract}
In \cite{MbPe}, Mbombo and Pestov prove that the group of isometries of the generalized Urysohn space of density $\kappa$, for uncountable $\kappa$ such that $\kappa^{<\kappa}=\kappa$, is not a universal topological group of weight $\kappa$. We investigate automorphism groups of other uncountable ultrahomogeneous structures and prove that they are rarely universal topological groups for the corresponding classes. Our list of uncountable ultrahomogeneous structures includes random uncountable graph, tournament, linear order, partial order, group. That is in contrast with similar results obtained for automorphism groups of countable (separable) ultrahomogeneous structures.

We also provide a more direct and shorter proof of the Mbombo-Pestov's result.
\end{abstract}
\maketitle
\section*{Introduction}
In \cite{Us}, Uspenskij proved that the group of all isometries of the Urysohn universal metric space with pointwise convergence topology is a universal topological group of countable weight. The proof consists of two steps. First, realizing that any topological group (of countable weight) is isomorphic to a subgroup of the group of isometries of some (separable) metric space. Second and most important, that every group of isometries of some separable metric space with pointwise convergence topology embeds into the group of isometries of the Urysohn space.

The latter fact was then observed in place of many other ultrahomogeneous (discrete) structures (we refer the reader to the next section for definitions). Let $M$ be a (discrete or metric) ultrahomogeneous structure and let $\mathrm{Aut}(M)$ denote the group of all automorphisms of $M$ equipped with the pointwise convergence topology. Is then $\mathrm{Aut}(M)$ the universal group for the class of groups that can be represented as automorphism groups of some substructure of $M$? More precisely, let $G$ be a topological group that is isomorphic to (a subgroup of) the group of automorphisms of some substructure $A\subseteq M$. Does $G$ homeomorphically embed into $\mathrm{Aut}(M)$?  E. Jaligot was probably the first who, in \cite {Ja}, formulated this problem in general and proved it for $M$ being the random tournament. The same was proved by Bilge and Jaligot in \cite{BiJa} when $M$ is a random $K_n$-free graph, for $n\geq 3$. Bilge and Melleray in \cite{BiMe} generalized this for ultrahomogeneous structures when the corresponding Fra\" iss\' e class has the free amalgamation property. Kubi\' s and Ma\v sulovi\' c in \cite{KuMa} gave another examples and formulated the problem in the language of category theory. Another metric versions are the following: Gao and Shao in \cite{GaSha} proved it for a universal and ultrahomogeneous separable $R$-ultrametric space, where $R\subseteq \Rea^+$ is some countable set of distances. And recently, Ben Yaacov in \cite{BY} when proving that the group of linear isometries of the Gurarij space is a universal Polish group proved that this linear isometry group is also universal for the corresponding class of groups (the universality as a Polish group then follows from the fact that every Polish group is isomorphic to a subgroup of linear isometries of some separable Banach space).

It is an open question whether there is actually a counterexample, i.e. a countable ultrahomogeneous structure $M$ and a substructure $A$ such that $\mathrm{Aut}(A)$ does not topologically embed into $\mathrm{Aut}(M)$.

For certains uncountable cardinals, namely those uncountable $\kappa$ such that $\kappa^{<\kappa}=\kappa$, one can have the (discrete and metric) Fra\" iss\' e theory as well. It means one can produce structures of cardinality $\kappa$ that are $\kappa$-ultrahomogeneous, i.e. any partial isomorphism between two substructures of cardinality strictly less than $\kappa$ extends to the full automorphism. Kat\v etov, in \cite{Kat}, proved the existence of the generalized Urysohn space $\mathbb{U}_\kappa$, for $\kappa$ as above, of weight $\kappa$.

Mbombo and Pestov in \cite{MbPe}, when checking whether $\mathrm{Iso}(\mathbb{U}_\kappa)$ might be the universal topological group of weight $\kappa$ (the existence of such groups of any uncountable weight is still not known), found out that situation changes at the higher cardinality. They proved that $\mathrm{Iso}(\mathbb{U}_\kappa)$ is not the universal topological group of weight $\kappa$, i.e. equivalently, there are metric spaces $X$ of density at most $\kappa$ such that $\mathrm{Iso}(X)$ does not topologically embed into $\mathrm{Iso}(\mathbb{U}_\kappa)$.

In this paper, we focus on the general universality problem in the ``uncountable Fra\" iss\' e theoretic setting", as considered by Mbombo and Pestov in the particular case of the uncountable Urysohn space. We show that while in the countable case the norm is that the automorphism group of an ultrahomogeneous structure is universal for the corresponding class of groups, in the uncountable case the norm seems to be the opposite - it is probably very rarely universal.

Let us list here some particular structures for which we can prove it.
\begin{thm}\label{non-uni}
Let $\kappa$ be an uncountable cardinal such that $\kappa^{<\kappa}=\kappa$. Let $M$ be one of the following strcutures: random graph of cardinality $\kappa$, random partial order of cardinality $\kappa$, random linear order of cardinality $\kappa$, random tournament of cardinality $\kappa$ and random group of cardinality $\kappa$. Then $\mathrm{Aut}(M)$ is not universal. There is a substructure $A\subseteq M$ such that $\mathrm{Aut}(A)$ does not continuously embed into $\mathrm{Aut}(M)$.
\end{thm}
We also provide a shorter and more direct proof of the Mbombo-Pestov's result.
\section{Preliminaries}
Let us explain more precisely the problem from the introduction here. The set-up is the following: let $L$ be some fixed countable language and let $M$ be a first-order $L$-structure, either countable discrete or separable metric (as is the case with the Urysohn space), that is ultrahomogeneous. These are precisely (metric) Fra\" iss\' e limits of (metric) amalgamation classes. We refer the reader to \cite{Ho} for information about Fra\" iss\' e theory, to \cite{BYFr} for information about metric Fra\" iss\' e theory and to \cite{Mac} for a survey on ultrahomogeneous structures.

The property characterizing these structures is the following. For any two finite substructures $A,B\subseteq M$ such that $A$ embeds into $B$ via an embedding $\iota:A\hookrightarrow B$ there exists an embedding $\rho: B\hookrightarrow M$ such that $\rho\circ \iota=\mathrm{id}_A$. This property is also called the finite-extension property of $M$. For any $L$-structure $K$, let $\mathrm{Age}(K)$ be the set of its all finite substructures. The universality of $M$ then means that for every countable (separable) $L$-structure $A$ such that $\mathrm{Age}(A)\subseteq \mathrm{Age}(M)$ there is an embedding of $A$ into $M$. Ultrahomogeneity is the property that every finite partial isomorphism between two finite substructures of $M$ extends to a full automorphism of $M$.

Although there is no general theorem, apparently for any known example of countable (or separable metric) ultrahomogeneous structures $M$ and for any substructure $A\subseteq M$ one can find a subgroup $G_A\leq \mathrm{Aut}(M)$ such that $\mathrm{Aut}(A)$ and $G_A$ are topologically isomorphic. Even something stronger holds in all known cases: for any such substructure $A$ one can find an isomorphic copy $A'$ of $A$ inside $M$ such that every automorphism of $A'$ extends to an automorphism of $M$.\\

Consider and fix now an uncountable cardinal $\kappa$ such that $\kappa^{<\kappa}=\kappa$. Consistently, there is no such cardinal (if there is no innacessible cardinal and the generalized continuum hypothesis fails at every regular cardinal). On the other hand, under GCH every isolated cardinal has this property. The generalized Fra\" iss\' e theorem (standard version from \cite{Fr}), replacing $\aleph_0$ by $\kappa$, works well for $\kappa$ with this property. We shall call a structure $M$, of cardinality $\kappa$, $\kappa$-ultrahomogeneous if any partial isomorphism between two substructures of $M$ of cardinality strictly less than $\kappa$ extends to a full automorphism of $M$. As in the countable case, it follows that $M$ is then a Fra\" iss\' e limit of the class $\mathrm{Age}_{<\kappa}(M)$ of all substructures of $M$ of size strictly less than $\kappa$. Moreover, if $\mathrm{Age}_{<\kappa}(M)$ contains (isomorphic copies of) all structures of given type of cardinality strictly less than $\kappa$, then we shall also call $M$ a $\kappa$-universal structure. It then follows from $\kappa$-ultrahomogeneity that $M$ contains as a substructure every structure of given type of cardinality at most $\kappa$.

By $\mathcal{A}(M)$, let us denote the following class of groups: $\{\mathrm{Aut}(N):N\text{ is a substructure of }M\}$. We shall say that $\mathrm{Aut}(M)$ is a universal topological group for the class $\mathcal{A}(M)$ if every $G\in \mathcal{A}(M)$ is topologically isomorphic to a subgroup of $\mathrm{Aut}(M)$; in other words, if for every substructure $A\subseteq M$ (of arbitrary size) the group $\mathrm{Aut}(A)$ homeomorphically embeds into $\mathrm{Aut}(M)$. 

The conjecture for countable (resp. separable metric) ultrahomogeneous structures is that the corresponding automorphism group is always universal for the corresponding class of groups. Let us note that in some trivial cases this is true as well for uncountable ultrahomogeneous structures. For instance, the cardinal $\kappa$ as a $\kappa$-ultrahomogeneous structure of an empty language is clearly universal. Similarly, if $M$ is a $\kappa$-ultrahomogeneous structure in a language having at most unary relations such that $\mathrm{Age}_{<\kappa}(M)$ has free amalgamation, then $\mathrm{Aut}(M)$ is universal as the standard methods for proving universality in the countable case work well there. However, our aim in this article is to show that in the interesting cases the universality fails.\\

Let us conclude this section with some notation. If $M$ is some discrete structure, of whatever cardinality, $A\subseteq M$ is some finite subset, then by $\mathcal{N}^{\mathrm{Aut}(M)}_A$ we shall denote the basic open neighbourhood of the identity in $\mathrm{Aut}(M)$ consisting of those elements that fix $A$ pointwise (if $M=\Nat$ is a countable structure of an empty language, then we shall write, as usual, $S_\infty$ there instead of $\mathrm{Aut}(M)$). If $M$ is equipped with a metric, the only such case will be when $M$ is the (generalized) Urysohn space, then we shall write $\mathrm{Iso}(M)$ instead of $\mathrm{Aut}(M)$. Moreover, if $A\subseteq M$ is a finite subset and $\varepsilon>0$, then  $\mathcal{N}^{\mathrm{Iso}(M)}_{A,\varepsilon}$ is the basic open neighbourhood of the identity consisting of those elements that move elements from $A$ less than $\varepsilon$-far.
\section{Proofs of theorems}
We do not have a general characterization of $\kappa$-ultrahomogeneous structures such that their automorphism groups are not universal nor do we have a single general proof for Theorem \ref{non-uni} from the introduction. We still start with a fairly simple non-universality result that is formulated quite generally. However, the only natural example known to us that fits into this general scheme is the automorphism group of the random graph of cardinality $\kappa$ (Fra\" iss\' e limit of the class of all graphs of cardinality strictly less than $\kappa$). We need a simple definition at first.
\begin{defin}[Property A]
Let $M$ be an infinite structure in some language $L$. Let $X_1,X_2\subseteq M$ be two countably infinite disjoint subsets. We say that $M$ has property A if there exists a set of binary relations $(R_i)_i\subseteq L$ and an element $x\in M$ such that for every $y\in X_1$ there is $i$ such that $R_i(x,y)$, and on the other hand, for no $y\in X_2$ there is $i$ such that $R_i(x,y)$.
\end{defin}
\noindent {\it Example 1}. If $M$ is a random graph of cardinality $\kappa$ and $X_1,X_2\subseteq M$ are two disjoint countable subsets, then there exists $x\in M$ such that $x$ is connected by an edge to every element of $X_1$ and to no element of $X_2$.\\

\noindent {\it Example 2}. If $M$ is a random tournament of cardinality $\kappa$ (recall that tournament is an oriented graph where every two vertices are connected by an edge) and $X_1,X_2\subseteq M$ are two disjoint countable subsets, then there exists $x\in M$ such that there is an oriented edge going from $x$ to every element of $X_1$ and such that there is an oriented edge going from every element of $X_2$ to $x$.

\begin{thm}
Let $M$ be a $\kappa$-universal and $\kappa$-ultrahomogeneous structure having property A. Then $S_\infty$ does not continuously embed into $\mathrm{Aut}(M)$.
\end{thm}
\begin{proof}
Suppose that $S_\infty$ does embed through the continuous embedding $e:S_\infty \hookrightarrow \mathrm{Aut}(M)$. 

We shall now inductively produce two disjoint countable subsets $\{m_1,m_2,\ldots\},\{n_1,n_2,\ldots\}$ of $M$ and elements $\phi_1,\phi_2,\ldots \in e[S_\infty]$. At the first step, let $\phi_1\in \mathrm{Aut}(M)$ be the image of $f_1\in S_\infty$ via $e$, where $f_1$ is arbitrary such that it fixes $1$, i.e belongs to $\mathcal{N}^{S_\infty}_{\{1\}}$. Pick an arbitrary element $m_1\in M$ such that $\phi_1(m_1)\neq m_1$. Clearly, such an element in $M$ exists since $\phi_1$ is not the identity. Denote $\phi_1(m_1)$ as $n_1$.

Suppose we have already found elements $m_1,\ldots,m_{l-1}$ and $n_1,\ldots,n_{l-1}$. Let $\phi_l\in \mathrm{Aut}(M)$ be the image of $f_l\in S_\infty$ via $e$, where $f_l$ is arbitrary such that it fixes the elements $1,\ldots,l$, i.e belongs to $\mathcal{N}^{S_\infty}_{\{1,\ldots,l\}}$, such that there exists $m_l\in M\setminus \{m_1,n_1,\ldots,m_{l-1},n_{l-1}\}$ such that $\phi_l(m_l)\notin \{m_l\}\cup \{m_1,n_1,\ldots,m_{l-1},n_{l-1}\}$. Such $\phi_l$ and $m_l$ exist. Otherwise, every element $\phi\in e[\mathcal{N}^{S_\infty}_{\{1,\ldots,l\}}]$ would move only the elements from the set $\{m_1,n_1,\ldots,m_{l-1},n_{l-1}\}$ and fix the rest, which is not possible since $\mathcal{N}^{S_\infty}_{\{1,\ldots,l\}}$ is uncountable. Denote $\phi_l(m_l)$ as $n_l$.

When the inductive construction is finished we use the property A of $M$ and find an element $x\in M$ such that we have a set of binary relations $(R_i)_i$ such that 
\begin{enumerate}
\item for every $l$ there is $i$ such $R_i(x,m_l)$
\item for no $l$ there is $i$ such that $R_i(x,n_l)$
\end{enumerate}
Let $\mathcal{N}^{\mathrm{Aut}(M)}_{\{x\}}$ be the basic open neighbourhood of the identity in $\mathrm{Aut}(M)$ consisting of those automorphisms of $M$ that fix $x$. Since $e$ is a continuous embedding, $\mathcal{N}=e^{-1}[\mathcal{N}^{\mathrm{Aut}(M)}_{\{x\}}]$ is an open neighbourhood of the identity in $S_\infty$. Let $l$ be such that $\mathcal{N}^{S_\infty}_{\{1,\ldots,l\}}\subseteq \mathcal{N}$. It follows that $e(f_l)=\phi_l\in e[\mathcal{N}^{S_\infty}_{\{1,\ldots,l\}}]\subseteq \mathcal{N}^{\mathrm{Aut}(M)}_{\{x\}}$. However, $\phi_l(m_l)=n_l$. According to $(1)$ there is some $i$ such that $R_i(x,m_l)$ and since $\phi_l$ is an automorphism fixing $x$, we get that $R_i(x,n_l)$ which contradicts $(2)$.
\end{proof}
\begin{cor}
Let $M$ be the random graph of cardinality $\kappa$. Then $\mathrm{Aut}(M)$ is not a universal topological group for the class $\mathcal{A}(M)$.
\end{cor}
\begin{proof}
By Example 1 above, $M$ has property A and thus by the previous theorem, $S_\infty$ does not continuously embed into $\mathrm{Aut}(M)$. Let $A\subseteq M$ be some countable clique, then $\mathrm{Aut}(A)=S_\infty$ and the statement follows.
\end{proof}
\begin{thm}
Let $M$ be a $\kappa$-universal and $\kappa$-ultrahomogeneous structure that contains as a substructure the $\kappa$-universal and $\kappa$-ultrahomogeneous linear order. Then $\mathrm{Aut}(M)$ is not a universal topological group for the class $\mathcal{A}(M)$.
\end{thm}
\begin{proof}
Suppose that it is universal. Then since $M$ contains the $\kappa$-universal and $\kappa$-ultrahomogeneous linear order, the automorphism group of this linear order continuously embeds into $\mathrm{Aut}(M)$. It follows that we may assume that $M$ is the $\kappa$-universal and $\kappa$-ultrahomogeneous linear order and it suffices to reach the contradiction by showing that $\mathrm{Aut}(\Rat,<)$ does not continuously embed into $\mathrm{Aut}(M)$.

Suppose it does embed and let $e:\mathrm{Aut}(\Rat,<)\hookrightarrow \mathrm{Aut}(M)$ be the continuous embedding.

For every $q\in \Rat$, let us denote by $s_q\in \mathrm{Aut}(\Rat)$ the shift by $q$, i.e. for every $h\in \Rat$ we have $s_q(h)=h+q$. Pick some $q\in \Rat^+$ and find and fix some $x\in M$ such that $e(s_q)(x)\neq x$. Let us show that then for every $p\in \Rat$ we have $e(s_p)(x)\neq x$. Indeed, first we can find $r\in \Rat$ and $k_p,k_q\in \Int$ such that $k_p\cdot r=p$ and $k_q\cdot r=q$. Consequently, $(s_r)^{k_q}=s_q$ and $(s_r)^{k_p}=s_p$. Thus $e(s_r)^{k_q}(x)=e(s_q)(x)\neq x$ and we must also have that $e(s_r)(x)\neq x$. It follows that also $e(s_p)(x)=e(s_r)^{k_p}(x)\neq x$.

By $\mathcal{I}$ we shall denote the set $\{z\in M: \exists q_L,q_R\in \Rat~ (e(s_{q_L})(x)\leq z\leq e(s_{q_R})(x))\}$.

We now need a series of claims.
\begin{claim}\label{claim1}
For every $q\in \Rat$ and for every $z\in \mathcal{I}$ we have $e(s_q)(z)\neq z$.
\end{claim}
Let $q\in \Rat$ and $z\in \mathcal{I}$ be given. By definition, there are $q_L,q_R\in \Rat$ such that $e(s_{q_L})(x)\leq z\leq e(s_{q_R})(x)$. Let $q_0\in \Rat$ be such that for some $k_L,k_R,k\in \Int$ we have $k\cdot q_0=q,k_L\cdot q_0=q_L,k_R\cdot q_0=q_R$.
Since $e(s_{q_0})^k(x)=e(s_q)(x)\neq x$ we have that $e(s_{q_0})(x)\neq x$.
Moreover, $x_L=e(s_{q_0})^{k_L}(x)\leq z\leq e(s_{q_0})^{k_R}(x)=x_R$. If $x_L=z=x_R$, then clearly $e(s_{q_0})(z)\neq z$, thus $e(s_q)(z)\neq z$. Otherwise, $e(s_{q_0})^{k_R-k_L}(x_L)=x_R>x_L$, thus $e(s_{q_0})^{k_R-k_L}(z)>x_R$, i.e. $e(s_{q_0})(z)\neq z$, thus again $e(s_q)(z)\neq z$.

\begin{claim}
For every $z\in \mathcal{I}$ there exists a finite set $F\subseteq \Rat$ such that $\forall f\in \mathcal{N}^{\mathrm{Aut}(\Rat)}_F~ (e(f)(z)=z)$.
\end{claim}
Otherwise, the preimage $e^{-1}[\mathcal{N}^{\mathrm{Aut}(M)}_{\{z\}}]$ of the open neighbourhood of the identity in $\mathrm{Aut}(M)$ would not be open in $\mathrm{Aut}(\Rat)$, contradicting the continuity of $e$.\\

For every finite $F\subseteq \Rat$, by $\mathcal{I}_F$ we shall denote the set $\{z\in \mathcal{I}: \forall f\in \mathcal{N}^{\mathrm{Aut}(M)}_F ~ (e(f)(z)=z)\}$. It is easy to see that each such $\mathcal{I}_F$ is closed (in the order topology). Also, we choose some enumeration $\{q_1,q_2,\ldots\}$ of the rationals, and then $\mathcal{I}_{\{q_1,\ldots,q_n\}}$ shall be denoted simply by $\mathcal{I}_n$. It follows that $\mathcal{I}=\bigcup_n \mathcal{I}_n$.

\begin{claim}
For every $n$, we have $\mathcal{I}_n\neq \mathcal{I}$.
\end{claim}
Suppose otherwise and let $n$ be such that $\mathcal{I}_n=\mathcal{I}$. For every $q\in \Rat$, by $\mathcal{I}_{n,q}$ we shall denote the set $\mathcal{I}_{\{q_1+q,\ldots,q_n+q\}}$. Since $\mathcal{N}^{\mathrm{Aut}(\Rat)}_{\{q_1+q,\ldots,q_n+q\}}\circ s_q=s_q\circ \mathcal{N}^{\mathrm{Aut}(\Rat)}_{\{q_1,\ldots,q_n\}}$ and $\mathcal{N}^{\mathrm{Aut}(\Rat)}_{\{q_1,\ldots,q_n\}}\circ s_{-q}=s_{-q}\circ \mathcal{N}^{\mathrm{Aut}(\Rat)}_{\{q_1+q,\ldots,q_n+q\}}$, we can deduce that $\mathcal{I}_{n,q}=e(s_q)[\mathcal{I}_n]$. Indeed, let $m\in e(s_q)[\mathcal{I}_n]$ and $f\in \mathcal{N}^{\mathrm{Aut}(\Rat)}_{\{q_1+q,\ldots,q_n+q\}}$ be arbitrary. Then there are $m'\in \mathcal{I}_n$ and $f'\in \mathcal{N}^{\mathrm{Aut}(\Rat)}_{\{q_1,\ldots,q_n\}}$ such that $e(s_q)(m')=m$ and $f\circ s_q=s_q\circ f'$. Thus $e(f)(m)=e(f\circ s_q)(m')=e(s_q\circ f')(m')=e(s_q)(m')=m$. The other inclusion is proved analogously. It follows that for every $q\in \Rat$ we have $\mathcal{I}_{n,q}=\mathcal{I}$. Now, the pointwise stabilizer of $\mathcal{I}$, the set $F_\mathcal{I}=\{\phi\in \mathrm{Aut}(M):\forall z\in \mathcal{I}~ (\phi(z)=z)\}$, is closed in $\mathrm{Aut}(M)$. So the preimage $e^{-1}[F_\mathcal{I}]$ must be closed in $\mathrm{Aut}(\Rat)$. It follows that the set $\{f\in \mathrm{Aut}(\Rat): \forall q\in \Rat ~(f\notin \mathcal{N}^{\mathrm{Aut}(\Rat)}_{\{q_1+q,\ldots,q_n+q\}})\}\supseteq e^{-1}[\mathrm{Aut}(M)\setminus F_\mathcal{I}]$ must have a non-empty interior in $(\mathrm{Aut}(\Rat),<)$. However, it is easy to check that it is nowhere dense. This contradiction finishes the proof of the claim.

\begin{claim}\label{helpclaim}
For no $n$ there exists a non-degenerated closed interval $I\subseteq \mathcal{I}$ such that $I\cap \mathcal{I}_n$ is a (possibly empty) subset of the end-points of $I$, $I\subseteq \mathcal{I}_{n+1}$ and for every $f\in \mathcal{N}^{\mathrm{Aut}(\Rat)}_{\{q_1,\ldots,q_n\}}$ we have $e(f)[I]\subseteq I$ (i.e. for every such $I$ we claim that $e[\mathcal{N}^{\mathrm{Aut}(\Rat)}_{\{q_1,\ldots,q_n\}}]$ is not contained in the set-wise stabilizer of $I$).
\end{claim}

Let us prove the claim. The proof is similar to the proof of the previous claim. Suppose that there exists such $n$. Suppose that $I=[x_1,x_2]$, for some $x_1<x_2\in \mathcal{I}$.

Let $f\in \mathcal{N}^{\mathrm{Aut}(\Rat)}_{\{q_1,\ldots,q_n\}}$ be such that $e(f)(z)\neq z$ for some $z\in I\setminus \{x_1,x_2\}$. Notice  that $f(q_{n+1})\neq q_{n+1}$.

Let $I^\Rat_f$ denote the set $\{q\in \Rat: \exists n_q\in \Int ~(f^{(n_q)}(q_{n+1})=q)\}$. Similarly, let $I^M_f\subseteq I$ be the set $\{m\in I: \exists n_m\in \Int ~(e(f^{(n_m)})(z)=m)\}$.

For every $n\in \Int$ and $m\in I^M_f$ we have that $e(f^{(n)})(m)\neq m$ (the proof is analogous to the proof of Claim \ref{claim1}). Moreover, since for every $q\in I^\Rat_f$ we have $\mathcal{N}^{\mathrm{Aut}(\Rat)}_{\{q_1,\ldots,q_n,q\}}\circ (f^{(n_q)})=(f^{(n_q)})\circ \mathcal{N}^{\mathrm{Aut}(\Rat)}_{\{q_1,\ldots,q_n,q_{n+1}\}}$, we can conclude that $I^M_f\subseteq\mathcal{I}_{\{q_1,\ldots,q_n,q\}}$. Indeed, let $q\in I^\Rat_f$, $g\in \mathcal{N}^{\mathrm{Aut}(\Rat)}_{\{q_1,\ldots,q_n,q\}}$ and $m\in I^M_f$ be arbitrary. Then there exist $n_m\in \Int$ such that $e(f^{(n_m)})(z)=m$ and $g'\in \mathcal{N}^{\mathrm{Aut}(\Rat)}_{\{q_1,\ldots,q_n,q_{n+1}\}}$ such that $g\circ f^{(n_q)}=f^{(n_q})\circ g'$.  So we have $$e(g)(m)=e(g\circ f^{(n_m}))(z)=e(g\circ f^{(n_q)}\circ f^{(n_m-n_q)})(z)=$$ $$e(f^{(n_q})\circ g'\circ f^{(n_m-n_q)})(z)=e(f^{(n_q)}\circ f^{(n_m-n_q)})(z)=m$$ where we used that $e(g'\circ f^{(n_m-n_q)})(z)=e(f^{(n_m-n_q)})(z)$ since $e(f^{(n_m-n_q)})(z)\in \mathcal{I}_{n+1}$.

Now as before, we argue that the pointwise stabilizer of $I^M_f$, the set $F_{I^M_f}=\{\phi\in \mathrm{Aut}(M):\forall m\in I^M_f ~(\phi(m)=m)\}$, is closed in $\mathrm{Aut}(M)$. So the preimage $e^{-1}[F_{I^M_f}]$ must be closed in $\mathrm{Aut}(\Rat)$. It follows that the set $\{f\in \mathrm{Aut}(\Rat): \forall q\in I^\Rat_f ~(f\notin \mathcal{N}^{\mathrm{Aut}(\Rat)}_{\{q_1,\ldots,q_n,q\}})\}\supseteq \mathrm{Aut}(\Rat)\setminus e^{-1}[F_{I^M_f}]$ must have a non-empty interior. However, as before, it is easy to check that it is nowhere dense. This contradiction finishes the proof of the claim.\\

We are now ready to finish the proof. At first, we construct by induction a sequence of nested (open) intervals $\mathcal{I}=I_0\supseteq I_1\supseteq I_2\supseteq \ldots$ with the following properties (for $n\geq 1$):
\begin{enumerate}
\item $I_n\subseteq I_{n-1}\cap (\mathcal{I}\setminus \mathcal{I}_n)$.
\item For every $z\in I_n$ there exists $f\in \mathcal{N}^{\mathrm{Aut}(\Rat)}_{\{q_1,\ldots,q_n\}}$ such that $e(f)(z)\neq z$.
\item If $z_1<z_2<z_3\in \mathcal{I}$ and $z_1,z_3\in I_1$, then also $z_2\in I_1$.
\item $I_n$ is maximal with respect to the properties above.

\end{enumerate}
It is clear that we can find such $I_1$. Suppose we have already chosen $I_1\supseteq\ldots\supseteq I_n$. The only possible reason why we could not continue would be that $I_n\subseteq \mathcal{I}_{n+1}$, i.e. we could not find an interval from $I_n\cap (\mathcal{I}\setminus \mathcal{I}_{n+1})$. We shall argue, using Claim \ref{helpclaim}, that that is not possible. We reach the contradiction by showing that $I_n$ then satisfies the requirements for the interval $I$ in the statement of Claim \ref{helpclaim}. Suppose that not. Then there must be $f\in \mathcal{N}^{\mathrm{Aut}(\Rat)}_{\{q_1,\ldots,q_n\}}$ such that that for some $z\in I_n$ we have $e(f)(z)\notin I_n$. Suppose that $e(f)(z)>I_n$ (i.e. $\forall y\in I_n ~(e(f)(z)>y)$), the other case is similar. Denote $e(f)(z)$ by $y$. However, by the property $(4)$, there exists $I_n<y'\leq y$ such that $y'\in \mathcal{I}_n$, i.e. $e(f)(y')=y'$. This is a contradiction.

So suppose we have constructed the sequence of such intervals. We now choose two sequences of points $y_1\leq y_2\leq\ldots <\ldots \leq z_2\leq z_1$ with the following property: there is some sequence of natural numbers $m_1<m_2<\ldots$ such that for every $n$ we have $y_n<I_{m_n}<z_n$ and $y_n,z_n\in I_{m_{n-1}}$. When this is done, since $M$ is $\kappa$-ultrahomogeneous, we can find an element $z_\infty\in M$ such that for every $n$ we have $y_n<z_\infty <z_n$. It follows that $z_\infty\in \bigcap_n I_n$. Consider now the open neighbourhood of the identity $\mathcal{N}=\mathcal{N}^{\mathrm{Aut}(M)}_{\{z_\infty\}}$. Since $e$ is continuous there exists $j$ such that $e[\mathcal{N}^{\mathrm{Aut}(\Rat)}_{\{q_1,\ldots,q_{m_j}\}}]\subseteq \mathcal{N}$. However, since $z_\infty\in I_{m_j}$ and $I_{m_j}\cap \mathcal{I}_{m_j}=\emptyset$, there exists $f\in \mathcal{N}^{\mathrm{Aut}(\Rat)}_{\{q_1,\ldots,q_{m_j}\}}$ such that $e(f)(z_\infty)\neq z_\infty$. This contradiction finishes the proof of the theorem.
\end{proof}
\begin{cor}
Let $M$ be one of the following structures: $\kappa$-universal and $\kappa$-ultrahomogeneous linear order, $\kappa$-universal and $\kappa$-ultrahomogeneous partial order or $\kappa$-universal and $\kappa$-ultrahomogeneous tournament. Then $\mathrm{Aut}(M)$ is not a universal topological group for the class $\mathcal{A}(M)$.
\end{cor}
\begin{proof}
All the structures from the statement contain the $\kappa$-universal and $\kappa$-ultrahomogeneous linear order as a substructure. So the corollary follows from the previous theorem.
\end{proof}
We conclude our list of discrete structures by an example having just operations and no relations, namely a group.

Let $\mathbb{G}_\kappa$ be the $\kappa$-universal and $\kappa$-ultrahomogeneous group generated by $\kappa$-many generators, i.e. every group with at most $\kappa$-many generators is isomorphic with some subgroup of $\mathbb{G}_\kappa$ and any partial isomorphism between two subgroups of $\mathbb{G}_\kappa$ generated by strictly less than $\kappa$-many generators extends to a full automorphism of $\mathbb{G}_\kappa$. This is a Fra\" iss\' e limit of the class of all groups generated by strictly less than $\kappa$ generators.
\begin{thm}
The group $\mathrm{Aut}(\mathbb{G}_\kappa)$ is not a universal topological group for the class $\mathcal{A}(\mathbb{G}_\kappa)$.
\end{thm}
\begin{proof}
The free group of countably many generators, denoted here by $\mathbb{F}_\infty$, clearly embeds into $\mathbb{G}_\kappa$. Also, one can check that $S_\infty\leq \mathrm{Aut}(\mathbb{F}_\infty)$, thus it suffices to prove that there is no continuous embedding of $S_\infty$ into $\mathrm{Aut}(\mathbb{G}_\kappa)$.

Suppose there is and let us denote it again by $e$. Let us enumerate the generators of $\mathbb{G}_\kappa$ as $\{g_\alpha:\alpha<\kappa\}$ in such a way that for every $\alpha<\beta<\kappa$ we have that $g_\beta\notin \langle\{g_\gamma:\gamma\leq \alpha\}\rangle$. By induction, we shall find two disjoint countable sets $\{f_n:n\in \Nat\},\{h_n:n\in \Nat\}\subseteq \mathbb{G}_\kappa$ such that 
\begin{enumerate}
\item For every $m\in \Nat$ there will be $p_m\in \mathcal{N}^{S_\infty}_{\{1,\ldots,m\}}$ and $n$ such that $e(p_m)(f_n)=h_n$.
\item There will be an element $g\in \mathbb{G}_\kappa$ such that $\forall n ~(g\cdot f_n=f_n\cdot g\wedge g\cdot h_n\neq h_n\cdot g)$.
\end{enumerate}
This suffices for reaching a contradiction. Indeed, let $\mathcal{N}=\mathcal{N}^{\mathrm{Aut}(\mathbb{G}_\kappa)}_{\{g\}}$. Then there must be $m\in \Nat$ such that $\mathcal{N}^{S_\infty}_{\{1,\ldots,m\}}\subseteq e^{-1}[\mathcal{N}]$. Then $e(p_m)\in \mathcal{N}$, thus $e(p_m)(g)=g$. However, since $e(p_m)(f_n)=h_n$ (for the appropriate $n$) and $f_n\cdot g=g\cdot f_n$, we must have $h_n\cdot g=e(p_m)(f_n\cdot g)=e(p_m)(g\cdot f_n)=g\cdot h_n$, a contradiction.\\

Let $\alpha_1=\min\{\alpha<\kappa: \exists p\in \mathcal{N}^{S_\infty}_{\{1\}}~ (e(p)(g_\alpha)\neq g_\alpha)\}$. Let $p_1\in \mathcal{N}^{S_\infty}_{\{1\}}$ be the corresponding element of $S_\infty$ such that $e(p_1)(g_{\alpha_1})\neq g_{\alpha_1}$. If $e(p_1)(g_{\alpha_1})\neq g_{\alpha_1}^k$, for some $k\in \Int$, then we set $f_1=g_{\alpha_1}$ and $h_1=e(p_1)(f_1)$. Otherwise, assuming without loss of generality that $\forall p\in S_\infty ~(e(p)(g_1)=g_1)$, i.e. $\alpha_1>1$, we set $f_1=g_{\alpha_1}\cdot g_1$ and $h_1=e(p_1)(f_1)$.

Suppose we have found the appropriate $f_1,h_1,\ldots,f_{l-1},h_{l-1}\in \mathbb{G}_\kappa$. Let $\alpha_l=\min \{\alpha<\kappa: \exists p\in \mathcal{N}^{S_\infty}_{\{1,\ldots,l\}} \forall i<l ~(e(p)(f_i)=f_i\wedge e(p)(h_i)=h_i\wedge e(p)(g_\alpha)\neq g_\alpha)\}$. Let $p_l$ be the corresponding element of $S_\infty$ such that $\forall i<l ~(e(p_l)(f_i)=f_i\wedge e(p_l)(h_i)=h_i\wedge e(p_l)(g_\alpha)\neq g_\alpha)$. If $e(p_l)(g_{\alpha_l})\notin \langle \{f_1,\ldots,f_{l-1},g_{\alpha_l}\}\rangle$, then we set $f_l=g_{\alpha_l}$ and $h_l=e(p_l)(f_l)$. Otherwise, we set $f_l=g_{\alpha_l}\cdot g_1$ and $h_l=e(p_l)(f_l)$.

When the induction is finished, we use the extension property of $\mathbb{G}_\kappa$ to find an element $g\in \mathbb{G}_\kappa$ satisfying  $\forall n ~(g\cdot f_n=f_n\cdot g\wedge g\cdot h_n\neq h_n\cdot g)$. This finishes the proof.
\end{proof}

We now provide a shorter and more direct proof of the result from \cite{MbPe} that $\mathrm{Iso}(\Ur_\kappa)$, for $\kappa^{<\kappa}=\kappa$, is not a universal topological group of weight $\kappa$. The proof is in the same spirit as the results above. Let us recall that a function $f:X\rightarrow \Rea^+$, where $X$ is a metric space, is called Kat\v etov if $\forall x,y\in X$ we have $|f(x)-f(y)|\leq d_X(x,y)\leq f(x)+f(y)$. The natural interpretation of such a function is to view it as a prescription of distances from some new imaginary point to the points of $X$. The generalized Urysohn space $\Ur_\kappa$ is characterized by the property that all Kat\v etov functions defined on subsets of $\Ur_\kappa$ of cardinality strictly less than $\kappa$ are realized by some point in $\Ur_\kappa$.
\begin{thm}
$S_\infty$ does not continuously embed into $\mathrm{Iso}(\Ur_\kappa)$. In particular, $\mathrm{Iso}(\Ur_\kappa)$ is not a universal topological group of weight $\kappa$.
\end{thm}
\begin{proof}
Suppose it does and let $e:S_\infty\hookrightarrow \mathrm{Iso}(\Ur_\kappa)$ be the continuous embedding. First we claim that that there exist $0<\varepsilon$, $n_0\in \Nat$ and elements $f_n\in \mathcal{N}^{S_\infty}_{\{1,\ldots,n\}}$, $x_n\in \Ur_\kappa$, for every $n\geq n_0$, such that $\forall n\geq n_0 ~(\varepsilon\leq d(x_n,e(f_n)(x_n))\leq 2\varepsilon)$.

Let us first argue for the lower bound. If there were no such $\varepsilon$ then for every $\delta>0$ there would exist $n_\delta\in \Nat$ such that for every $f\in \mathcal{N}^{S_\infty}_{\{1,\ldots,n_\delta\}}$ we would have $d_{\sup}(e(f),\mathrm{id})<\delta$, where $d_{\sup}$ is the supremum metric on $\mathrm{Iso}(\Ur_\kappa)$ (not compatible with the standard topology); i.e. $\forall x\in \Ur_\kappa ~(d(x,e(f)(x))<\delta)$. However, one could then argue that the topology on $e[S_\infty]$ induced by $d_{\sup}\upharpoonright e[S_\infty]$ agrees with the standard topology on $e[S_\infty]$ and is induced by a two-sided invariant metric $d_{\sup}\upharpoonright e[S_\infty]$, which is a contradiction (recall that $S_\infty$ does not admit a compatible complete left-invariant metric; see \cite{Gao} for example). So let us fix $\varepsilon>0$ such that for every $n$ there exist $f_n\in \mathcal{N}^{S_\infty}_{\{1,\ldots,n\}}$ and $x'_n\in \Ur_\kappa$ such that $d(x'_n,e(f_n)(x'_n))>\varepsilon$. Clearly, for every $z\in \Ur_\kappa$ if $n$ is large enough then $d(z,e(f_n)(z))<2\varepsilon$. That follows from the continuity of the embedding $e$. Now for every large enough $n$ (greater than some $n_0$) we can find the desired $x_n\in \Ur_\kappa$, i.e. $\varepsilon\leq d(x_n,e(f_n)(x_n))\leq 2\varepsilon$, somewhere on the geodesic segment connecting $z$ and $x'_n$. For every $n\geq n_0$, let us denote $e(f_n)(x_n)$ by $y_n$.

We now find an infinite subset $\{n_0,n_1,\ldots\}\subseteq \Nat\setminus \{1,\ldots,n_0-1\}$ so that we can define a Kat\v etov function $F:\{x_{n_i},y_{n_i}:i\in \Nat\}\rightarrow \Rea^+$ such that $\forall i ~(|F(x_{n_i})-F(y_{n_i})|\geq\varepsilon/4)$. The statement of the theorem then follows. Indeed, it follows that there exists an element $x_F\in \Ur_\kappa$ realizing $F$. We then consider the open neighbourhood of the identity $\mathcal{N}=\mathcal{N}^{\mathrm{Iso}(\Ur_\kappa)}_{\{x_F\},\varepsilon/4}$. Since $e$ is continuous there exists $i$ such that $e[\mathcal{N}^{S_\infty}_{\{1,\ldots,n_i\}}]\subseteq \mathcal{N}$. It follows that $d(x_F,e(f_{n_i})(x_F))<\varepsilon/4$. Since $e(f_{n_i})$ is an isometry we have $d(x_F,x_{n_i})=d(e(f_{n_i})(x_F),y_{n_i})$, thus $|d(x_F,x_{n_i})-d(x_F,y_{n_i})|=|F(x_{n_i})-F(y_{n_i})|<\varepsilon/4$, a contradiction. It remains to find such an infinite subset.\\

We claim that we may assume that the set $\{x_i,y_i:i\geq n_0\}$ has bounded diameter. Suppose that not. Then we shall find another sequence $(z_n)_{n\geq n_0}$ with bounded diameter and such that for every $n\geq n_0$ we have $\varepsilon\leq d(z_n,e(f_n)(z_n))\leq 2\varepsilon$. Set $z_{n_0}=x_{n_0}$. For any $n>n_0$, we may suppose that $d(z_{n_0},x_n)>3\varepsilon$, and also that $d(z_{n_0},e(f_n)(z_{n_0}))\leq 2\varepsilon$ since this must hold true for $n$ large enough because of continuity of $e$. Let again $n>n_0$. If $d(z_{n_0},e(f_n)(z_{n_0}))\geq \varepsilon$ then we set $z_n=z_{n_0}$. Otherwise, since $d(z_{n_0},x_n)=d(e(f_n)(z_{n_0}),y_n)$ we must have $|d(z_{n_0},x_n)-d(z_{n_0},y_n)|<\varepsilon$. Using the extension property of $\Ur_\kappa$ we can find an element $z_n\in \Ur_\kappa$ such that $d(z_n,x_n)=d(z_{n_0},x_n)$, $d(z_n,y_n)=d(z_{n_0},x_n)-\varepsilon$ and $d(z_n,z_{n_0})=2\varepsilon$. It is easily checked that all triangle inequalities are satisfied. However, then we have that $d(z_n,x_n)=d(e(f_n)(z_n),y_n)$ and since  $d(z_n,y_n)=d(z_{n_0},x_n)-\varepsilon$ we must have $d(z_n,e(f_n)(z_n))\geq \varepsilon$. If $d(z_n,e(f_n)(z_n))>2\varepsilon$ then replace $z_n$ by an element $z'_n$ lying on the geodesic segment connecting $z_{n_0}$ and $z_n$ so that $\varepsilon\leq d(z'_n,e(f_n)(z'_n))\leq 2\varepsilon$.\\

Thus we now assume that $\{x_i,y_i:i\geq n_0\}$ has bounded diameter $D$. Using the Ramsey theorem we can further refine this set to the set $\{x_{n_i},y_{n_i}:i\geq 0\}$, for some subsequence $n_0<n_1<\ldots$, so that for every $i$ there is no $j$ such that $d(x_{n_i},y_{n_j})\leq \varepsilon/4$ (define a function $\rho:[\Nat]^2\rightarrow \{0,1\}$ such that $\rho(i,j)=0$ iff $d(x_{n_i},y_{n_j})>\varepsilon/4$ and $d(x_{n_j},y_{n_i})>\varepsilon/4$, then use the Ramsey theorem to find the homogeneous set in $0$, observe that there cannot be an infinite homogeneous set in $1$). Now, let $F(x_{n_i})=2D$, for every $i$, and we put $F(y_{n_i})=\min\{F(x_{n_j})+d(x_{n_j},y_{n_i}):j\in \Nat\}$ for every $i$. It is easy to check that $F$ is as desired. 
\end{proof}


\begin{thebibliography}{10}
\bibitem{BYFr}
I. Ben Yaacov, \emph{Fra\" iss\' e limits of metric structures}, preprint, arXiv:1203.4459 [math.LO], 2012
\bibitem{BY}
I. Ben Yaacov, \emph{The linear isometry group of the Gurarij space is universal}, Proc. Amer. Math. Soc. 142 (2014), no. 7, 2459-2467
\bibitem{BiJa}
D. Bilge, E. Jaligot, \emph{Some rigid moieties of homogeneous graphs}, Contrib. Discrete Math. 7 (2012), no. 2, 66-71
\bibitem{BiMe}
D. Bilge, J. Melleray, \emph{Elements of finite order in automorphism groups of homogeneous structures}, preprint, 2013
\bibitem{Fr}
R. Fra\" iss\' e, \emph{Sur certaines relations qui g\' en\' eralisent l'order des nombres rationnels}, C. R. Acad. Sci. Paris 237 (1953) 540-542
\bibitem{Gao}
S. Gao, \emph{Invariant Descriptive Set Theory}, CRC Press, 2009
\bibitem{GaSha}
S. Gao, C. Shao, \emph{Polish ultrametric Urysohn spaces and their isometry groups}, Topology Appl. 158 (2011), no. 3, 492-508
\bibitem{Ho}
W. Hodges, \emph{Model theory. Encyclopedia of Mathematics and its Applications, 42}, Cambridge University Press, Cambridge, 1993
\bibitem{Ja}
E. Jaligot, \emph{On stabilizers of some moieties of the random tournament}, Combinatorica 27 (2007), no. 1, 129-133
\bibitem{Kat}
M. Kat\v etov, \emph{On universal metric spaces}, General topology and its relations to modern analysis and algebra, VI (Prague, 1986),  323-330, 
Res. Exp. Math., 16, Heldermann, Berlin, 1988
\bibitem{KuMa}
W. Kubi\' s, D. Ma\v sulovi\' c, Kat\v etov functors, paper in preparation
\bibitem{Mac}
D. Macpherson, \emph{A survey of homogeneous structures}, Discrete Math. 311 (2011), no. 15, 1599-1634
\bibitem{MbPe}
B. Mbombo, V. Pestov, \emph{Subgroups of isometries of Urysohn-Kat\v etov metric spaces of uncountable density}, Topology Appl. 159 (2012), no. 9, 2490-2496
\bibitem{Us}
V. Uspenskij, \emph{On the group of isometries of the Urysohn universal metric space}, Comment. Math. Univ. Carolin. 31 (1990), no. 1, 181-182
\end{thebibliography}
\end{document}